\documentclass[noinfoline]{article}

\RequirePackage[OT1]{fontenc}
\RequirePackage[%lnms,
amsthm,amsmath%,natbib
]{imsart}

\usepackage{graphicx}
\usepackage{latexsym,amsmath}
\usepackage{amsmath,amsthm,amscd}
\usepackage{amsfonts}
\usepackage[psamsfonts]{amssymb}
\usepackage{enumerate}
\usepackage{cite}

\usepackage{url}
\usepackage{tocvsec2}

\usepackage{epstopdf}
\usepackage{float}
\usepackage{hyperref}

\startlocaldefs

\setattribute{keyword}{AMS}{AMS 2010 subject classifications:}

\theoremstyle{plain} 
\newtheorem{theorem}{Theorem}%[section]
\newtheorem{corollary}[theorem]{Corollary}

{Lemma}

\theoremstyle{definition} 

\theoremstyle{definition} 
\newtheorem{ex}[theorem]{Example}
\newtheorem*{ex*}{Example}
\theoremstyle{remark} 

\theoremstyle{remark} 

\newtheorem*{remark*}{Remark}
\newcommand{\al}{\alpha}

\newcommand{\la}{\lambda}

\newcommand{\be}{\beta}

\renewcommand{\th}{\theta}

\renewcommand{\P}{\operatorname{\mathsf{P}}} 
\newcommand{\E}{\operatorname{\mathsf{E}}}

\newcommand{\R}{\mathbb{R}}

\newcommand{\C}{\mathcal{C}}
\newcommand{\F}{\mathcal{F}}

\newcommand{\tPi}{{\tilde{\Pi}}}

\renewcommand{\le}{\leqslant}
\renewcommand{\ge}{\geqslant}

\endlocaldefs

\begin{document}

%\jobname.tex
%\today

\begin{frontmatter}

\title{Exact Rosenthal-type inequalities for \emph{p}\,=\,3, %$p=3$, 
and related results}
\runtitle{Exact Rosenthal-type inequalities}
%\date{\today}

% \author{\fnms{First}  \snm{Author}\corref{}\thanksref{t2}\ead[label=e1]{first@somewhere.com}},
%  \author{\fnms{Second} \snm{Author}\ead[label=e2]{second@somewhere.com}}
%  \and
%  \author{\fnms{Third}  \snm{Author}%
%  \ead[label=e3]{third@somewhere.com}%
%  \ead[label=u1,url]{http://www.foo.com}}
%
%  \thankstext{t2}{Footnote to the first author with the `thankstext' command.}

\begin{aug}
\author{\fnms{Iosif} \snm{Pinelis}\thanksref{t2}\ead[label=e1]{ipinelis@mtu.edu}
}
  \thankstext{t2}{Supported by NSA grant H98230-12-1-0237}
\runauthor{Iosif Pinelis}

%\affiliation{Michigan Technological University}

\address{Department of Mathematical Sciences\\
Michigan Technological University\\
Houghton, Michigan 49931, USA\\
E-mail: \printead[ipinelis@mtu.edu]{e1}}
\end{aug}

\begin{abstract} 
An exact Rosenthal-type inequality for the third absolute moments is given, as well as a number of related results. 
Such results are useful in applications to Berry--Esseen bounds. 
%One of the results of this paper is the inequality 
%$$\E(S-x)_+^3&\le\E(\si Z-x)_+^3+\sum\E(X_i)_+^3,$$ 
%where $x$ is any real number, $Z$ is a standard normal random variable (r.v.), $S:=\sum X_i$, and the $X_i$ are independent zero-mean r.v.'s, and and $\si:=\sqrt{\Var X}$.     
\end{abstract}

%\subjclass[2000]{60E15, 62G10, 62G15, 60G50, 62G35}
% 62G10    	Hypothesis testing
%  62G15    	Tolerance and confidence regions
%  60G50    	Sums of independent random variables; random walks
%   62G35    	Robustness
  
%
%\keywords{probability inequalities; Rade\-macher random variables; sums of independent random variables; Student's test; self-normalized sums}



%AMS 2000 subject classications: Primary 60E15; secondary 46B09.
%Keywords and phrases: probability inequalities, Rosenthal inequality,
%sums of independent random variables, martingales, concentration of mea-
%sure, separately Lipschitz functions, product spaces.

% 46B09   	Probabilistic methods in Banach space theory
%  60B11   	Probability theory on linear topological spaces [See also 28C20]
%   60E15   	Inequalities; stochastic orderings

\end{frontmatter}

\settocdepth{chapter}

%\tableofcontents 
%%%%%%%%%%%%%%%%%{\small\tableofcontents} 

\settocdepth{subsubsection}

\theoremstyle{plain} 
%\newtheorem{theorem}{Theorem}[section]
%\newtheorem{corollary}[theorem]{Corollary}
%\newtheorem*{main}{Main~Theorem}
%\newtheorem{lemma}{Lemma}[subsection]
%\newtheorem{proposition}[theorem]{Proposition}
%\newtheorem{conjecture}{Conjecture}
%\theoremstyle{definition} 
%\newtheorem{definition}[theorem]{Definition}
%\theoremstyle{definition} 
%\newtheorem{ex}{Example}
%\theoremstyle{remark} 
%\newtheorem{exer}{Exercise}
%\theoremstyle{remark} 
%\newtheorem{remark}[theorem]{Remark}b
%\newtheorem*{remark*}{Remark}
%\numberwithin{equation}{section}

%\eject

\section{Introduction, summary, and discussion}\label{intro} 
Let $X_1,\dots,X_n$ be independent random variables (r.v.'s), with the sum $S:=X_1+\dots+X_n$, such that 
for some real positive constant %$\si$ and 
$\be$ and all $i$ one has 
\begin{equation}\label{eq:BH conds}
	%X_i\le y,\quad 
	\E X_i\le0, \quad %\text{and}\quad 
	%0<
	\sum\E X_i^2\le1,%\si^2,
	\quad\text{and}\quad \sum\E(X_i)_+^3\le\be; %\tag{$\BH$ conds}
\end{equation}
as usual, we let $x_+:=0\vee x$ and $x_+^p:=(x_+)^p$ for all real $x$ and all real $p>0$. 

Consider the following class of functions:
\begin{align} \notag%\label{eq:F3}
\F^3&:=\{f\in\C^2\colon \text{$f$ and $f''$ are nondecreasing and convex}\} \notag\\
&=\{f\in\C^2\colon \text{$f$, $f'$, $f''$, $f'''$ are nondecreasing}\}, \label{eq:F3}
\end{align} 
where $\C^2$ denotes the class of all twice continuously differentiable functions 
$f\colon\R\to\R$ and $f'''$ denotes the right derivative of the convex function $f''$. 
For example, functions $x\mapsto a+b\,x+c\,(x-t)_+^\al$ and 
$x\mapsto a+b\,x+c\,e^{\la x}$ belong to $\F^3$ for all $a\in\R$, $b\ge0$, $c\ge0$, $t\in\R$, $\al\ge3$, and $\la\ge0$. 

\begin{remark*}% \label{rem:Jensen's}
If a r.v.\ $X$ has a finite expectation and a function $f\colon\R\to\R$ is in $\F^3$ or, more generally, is any convex function, then, by Jensen's inequality, $\E f(X)$ always exists in $(-\infty,\infty]$. 
\end{remark*}  

The main result of this note is 

\begin{theorem}\label{th:}
For any function $f\in\F^3$
\begin{equation}\label{eq:}
	\E f(S)\le\E f(Z)+\frac{f'''(\infty-)}{3!}\,\be, %\sum\E(X_i)_+^3,  
\end{equation}
where $Z$ is a standard normal r.v. %\ and $f'''$ is the left derivative of the convex function $f''$ 
%; as usual, we let $x_+:0\vee x$ and $x_+^p:=(x_+)^p$ for all real $x$ and all real $p>0$.  %
%(which is nondecreasing). 
Moreover, for each function $f\in\F^3$ the upper bound in \eqref{eq:} is exact, in the sense that it is equal to the supremum of $\E f(S)$ over all independent $X_i$'s satisfying conditions \eqref{eq:BH conds}. % equals the right-hand side of \eqref{eq:}. 
\end{theorem}
Of course, in the case when $f'''(\infty-)=\infty$, the inequality \eqref{eq:} is trivial. 
Theorem~\ref{th:} is based on the main result of \cite{pin-hoeff%,pin-hoeff-AIHP
}. 

%\begin{remark}\label{rem:}
%%
%\end{remark}
It follows immediately from Theorem~\ref{th:} that for all real $x$ 
\begin{align} 
	\E(S-x)_+^3&\le\E(Z-x)_+^3+\be. 	\label{eq:_+^3}
\intertext{If it is additionally assumed that $\E X_i=0$ for all $i$, then \eqref{eq:_+^3} in turn yields} 
	\E|S-x|^3&\le\E|Z-x|^3+\sum\E|X_i|^3; \label{eq:| |^3} 
\end{align}
moreover, one can similarly show that the upper bound in \eqref{eq:| |^3} is exact, for each real $x$;  
%\begin{equation}\label{eq:_+^3} 
%	\E(S-x)_+^3\le\E (Z-x)_+^3+\be; 
%\end{equation}
%if, moreover, it is assumed that $\E X_i\le0$ for all $i$, then 
%\begin{equation}\label{eq:| |^3} 
%	\E|S-x|^3\le\E|Z-x|^3+\sum\E|X_i|^3. 
%\end{equation}
the special case $x=0$ of \eqref{eq:| |^3} is also a special case of Rosenthal's inequality \cite{rosenthal}: 
\begin{equation}\label{eq:rosenthal}
	\E|S|^p\le c_p\,\big(1+\sum\E|X_i|^p\big), 
\end{equation}
for all $p\ge2$, where $c_p$ is a positive constant depending only on $p$ (inequality \eqref{eq:rosenthal} too needs the assumption that the $X_i$'s be zero-mean).  
In the case when $x=0$ and the $X_i$'s are symmetric, inequality \eqref{eq:| |^3} was obtained by Ibragimov and Sharakhmetov \cite{ibr-shar97}, who at that considered arbitrary real $p>2$. 
Besides, inequality \eqref{eq:| |^3} follows from Tyurin's result \cite[Theorem~2]{tyurinSPL}, which also implies \eqref{eq:_+^3} but with $\sum\E|X_i|^3$ in place of $\be$. 
More on Rosenthal-type inequalities and related results can be found, among other papers, in \cite{%rosenthal,
burk,MR0443034,pin-utev84,utev-extr,sibam,latala-moments,gine-lat-zinn,
%novak00,
ibr-sankhya,bouch-etal,pin12-2smooth%,novak05
}. 

Theorem~\ref{th:} admits 

\begin{corollary}\label{cor:}
For any $p\in(0,3)$ and any real $a>0$
\begin{align*}
	\E S_+^p&\le\frac{p^p(3-p)^{3-p}}{3^3}\,\frac{\E(Z+a)_+^3+\be}{a^{3-p}}; 
\end{align*}	 
%	\intertext{
in particular, taking here $(p,a)=(1,\frac{1746}{1000})$ and $(p,a)=(2,\frac{639}{1000})$, one obtains, respectively, the inequalities 
\begin{equation}\label{eq:cor}
	\E S_+\le0.514+0.0486\be\quad\text{and}\quad \E S_+^2\le0.555+0.232\be. 
\end{equation} 
\end{corollary}

One may compare the latter two bounds with the ``naive'' ones, obtained using the inequalities 
$(\E S_+)^2\le\E S_+^2\le\E S^2\le1$; here one may note that $\be$ will rather typically be small. 
One can similarly bound $\E(S-x)_+^p$ for any real $x$ and any $p\in(0,3)$. 
The first inequality in \eqref{eq:cor} can in fact be improved: 
\begin{equation}\label{eq:p=1}
	\E S_+\le\tfrac12, 
\end{equation}
which follows because $4u_+\le u^2+2u+1$ for all real $u$; the bound $\frac12$ on $\E S_+$ in \eqref{eq:p=1} is obviously attained when $\P(S=\pm1)=\frac12$. % is a Rademacher r.v. 

The case $p=3$ of Rosenthal-type inequalities, including the results stated above, is especially important in applications to Berry--Esseen bounds; see e.g.\ \cite{nonlinear}, Remark~3.4 in \cite{more-nonunif}, and the ``quick proofs'' of Nagaev's nonuniform Berry--Esseen bound in \cite{nonunif,more-nonunif}.  

\section{Proofs}\label{proofs}

\begin{proof}[Proof of Theorem~\ref{th:}]
Take indeed any $f\in\F^3$. 
Next, take any real $y>\be$ and introduce the r.v.'s 
\begin{equation*}
	X_{i,y}:=X_i\wedge y\quad\text{and}\quad S_y:=\sum_i X_{i,y}. 
\end{equation*}
Then the conditions \eqref{eq:BH conds} hold for the $X_{i,y}$'s in place of $X_i$. 
Also, %trivially, 
$X_{i,y}\le y$ for all $i$. 
So, by the main result of \cite{pin-hoeff}, 
\begin{align}
	\E f(S_y)&\le\E f\big(\sqrt{1-\be/y}\,Z+y\tPi_{\be/y^3}\big) \label{eq:Pin} \\ 
	&=\sum_{j=0}^\infty\E f\big(\sqrt{1-\be/y}\,Z+yj-\be/y^2\big)\frac{(\be/y^3)^j}{j!}e^{-\be/y^3}, \label{eq:sum} %\quad\text{for all }f\in\F^3,
\end{align}
where $\tPi_{\th}:=\Pi_{\th}-\E\Pi_{\th}=\Pi_{\th}-\th$ and $\Pi_{\th}$ is any r.v.\ which is independent of $Z$ and has the Poisson distribution with parameter $\th$, for any real $\th>0$. 
Moreover, by \cite[Proposition 2.3]{pin-hoeff}, for any given triple $(f,\be,y)\in\F^3\times(0,\infty)\times(0,\infty)$ with $y>\be$ the bound in \eqref{eq:Pin} is exact, in the sense that it is equal to the supremum of $\E f(S_y)$ 
over all independent $X_i$'s satisfying conditions \eqref{eq:BH conds}. 

Now let 
$$y\to\infty.$$ 
Then, by the monotone convergence theorem, 
\begin{equation}\label{eq:lim}
	\E f(S_y)\to\E f(S). 
\end{equation}
As was mentioned earlier, in the case when $f'''(\infty-)=\infty$ the inequality \eqref{eq:} is trivial. Consider now the case when $f'''(\infty-)<\infty$. Then, by a l'Hospital-type rule,  
$f(x)/x^3\to f'''(\infty-)/3!$ as $x\to\infty$, which also leads to 
%for each $q\in\{0,1,2,3\}$ 
%there is a finite positive real constant 
%$C$ such that 
$|f(x)|=O(1+|x|^3)$ over all real $x$ \big(for negative real $x$, one even has $|f(x)|=O(1+|x|)$, since $f$ is nondecreasing and convex; cf.\ e.g.\ \cite[Lemma~7]{asymm}\big). 
%$C_q$ such that $|f^{(q)}(x)|\le C_q\,(1+|x|^{3-q})$ for all real $x$, where $f^{(q)}$ denotes, as usual, the $q$th derivative of $f$. 
Therefore, by the dominated convergence theorem, 
\begin{equation*}
\begin{alignedat}{2}
	\E f\big(\sqrt{1-\be/y}\,Z+yj-\be/y^2\big)&\longrightarrow\E f(Z) && \text{\ \ if } j=0 \\
	\frac{\E f\big(\sqrt{1-\be/y}\,Z+yj-\be/y^2\big)}{y^3}&\longrightarrow f'''(\infty-)\,\frac{j^3}{3!}  && \text{\ \ if } j>0, 
\end{alignedat}	
\end{equation*} 
%\begin{equation}
%	\E f\big(\sqrt{1-\be/y}\,Z+yj-\be/y^2\big)
%	\left\{
%	\begin{alignedat}{2}
%	&\longrightarrow\E f(Z) && \text{\ \ if } j=0, \\ 
%	&\sim f'''(\infty-)(yj)^3/3!  && \text{\ \ if } j>0, 
%	\end{alignedat}
%	\right.
%\end{equation} 
and so, again by the dominated convergence theorem (say), the sum in \eqref{eq:sum} converges to $\E f(Z)+\frac{f'''(\infty-)}{3!}\,\be$. 
In view of \eqref{eq:Pin}--\eqref{eq:lim}, this proves the inequality \eqref{eq:}; the exactness of the bound in \eqref{eq:} follows from that of the bound in \eqref{eq:Pin}--\eqref{eq:sum}.  
\end{proof}

\begin{proof}[Proof of Corollary~\ref{cor:}]
This follows from \eqref{eq:_+^3}, since 
$\sup_{u\ge0}\frac{u^p}{(u+a)^3}=\frac{p^p(3-p)^{3-p}}{3^3a^{3-p}}%\,\frac1{a^{3-p}}
$ for any $p\in(0,3)$ and any real $a>0$. 
\end{proof}

\bibliographystyle{abbrv}
%\bibliographystyle{ims}
%\bibliography{are.citations}
%\bibliography{citat}

%\bibliography{citations}

\bibliography{C:/Users/Iosif/Dropbox/mtu/bib_files/citations12.13.12}
%\bibliography{C:/Users/iosif-home-2011/Dropbox/mtu/bib_files/citations}
%\bibliography{C:/Users/Iosif/Documents/mtu_home01-30-10/bib_files/citations}
%\bibliography{C:/Users/Iosif/Documents/mtu_home12-22-08/bib_files/citations}

\end{document}